\documentclass[12pt,a4paper]{article}
\usepackage{amsmath,amsthm,amsfonts,amssymb,amscd}
\usepackage[T1]{fontenc}
\usepackage[latin1]{inputenc}
\usepackage{amsfonts}
\usepackage{cite}
\usepackage{pdfpages,eso-pic,calc}
\usepackage{epstopdf}

\newtheorem{thm}{Th\'eor\`eme}[section]
\newtheorem{theo}{Theorem}[section]
\newtheorem{lema}{Lemma}[section]

\newtheorem{coro}[theo]{Corollary}

\theoremstyle{exemple}

\newtheorem{rk}[thm]{Remark}

\theoremstyle{definition}

\def\il{\int}

\def\calA{{\cal{A}}}

\def\R{\mathbb{R}}
\def\N{\mathbb{N}}

\def\calE{{\cal{E}}}
\def\calF{{\cal{F}}}

\newcommand{\Wa}{W^{\alpha/2,2}}

\newcommand{\Wo}{W_0^{\alpha/2,2}(\Omega)}

\newcommand{\cf}{\frac{\calA{(d,\alpha)}}{2}}

\newcommand{\V}{V}

\newcommand{\varp}{\varphi}
\newcommand{\lam}{\lambda}
\newcommand{\Lam}{\Lambda}
\newcommand{\Om}{\Omega}

\newcommand{\alp}{\alpha}

\newcommand{\phiok}{\varp_0^{(k)}}

\newcommand{\phiv}{\varp_0^{V}}
\newcommand{\phis}{\varp_0^{*}}
\newcommand{\lamv}{\lam_0^{V}}
\newcommand{\xiv}{\xi^{V}}
\newcommand{\phio}{\varp_0}
\newcommand{\lamos}{\lam_0^{*}}
\newcommand{\xis}{\xi^{*}}

\numberwithin{equation}{section}

\begin{document}

\title{Intrinsic ultracontractivity for fractional Schr\"{o}diner operators}

\author{ Mohamed Ali BELDI}

\date{}
\maketitle{}
\begin{abstract}We establish sharp pointwise estimates for the ground states of some singular fractional Schr\"odinger operators on relatively compact Euclidean subsets. The considered operators are of the type $(-\Delta)^{\alpha/2}|_\Om-V$, where $V\in L^1_{Loc}$ and $(-\Delta)^{\alpha/2}|_\Om$ is the fractional-Laplacien
on an open subset $\Om$ in $\R^d$ with zero exterior condition . The intrinsic ultracontractivity property for such operators is discussed as well and a sharp large time asymptotic for their heat
kernels is derived.
\end{abstract}

\section{Introduction}
This paper is to  study  intrinsic ultracontractivity for the Feynman-Kac semigroups generated
  by Schr\"{o}dinger   operators based on fractional Laplacians and obtain two sharp estimates of the
  first eigenfunction    of these operators, we use potentiel methods and Sobolev inequalities.
Let $\Om$ be a $C^{1,1}$ bounded domain in $\R^d$ containing the origin. Let
$L_0:=(-\Delta)^{\alpha/2}|_\Om,\ 0<\alpha<\min(2,d)$ be the fractional Laplacien on $\Om$ with zero
exterior condition in $L^2(\Om,dx)$.  It is well known  that $L_0$ has purely discrete spectrum
$0<\lam_0<\lam_1<\cdots<\lam_k\to\infty$ and that the associated semigroup $T_t:=e^{-tL_0}, t>0$ is
irreducible. Hence $L_0$ has a unique strictly positive normalized ground state $\varp_0$. Furthermore
Kulczycki proved in \cite{kul} that the semigroup $(T_t)_{ t>0}$ is intrinsically ultracontractive
(IUC for short) regardless the regularity of $\Om$. The latter property induces among others the large time
asymptotic for the heat kernel $p_t$ of  $e^{-tL_0},\ t>0$:
\begin{eqnarray}
p_t(x,y)\sim e^{-t\lam_0}\phio(x)\phio(y),\ {\rm on}\ \Om\times\Om.
\end{eqnarray}
Such type of estimates are very important in the sense that they give precise information on the local
behavior of the ground state and the heat kernel (for large $t$) as well as on their respective rates of
decay at the boundary.\\
  Set $G$  the Green's Kernel of $L_0$, that since $(T_t)_{ t>0}$  is IUC, then there is
a finite constant $C_G$, such that

                  \begin{eqnarray}
                  G(x,y)\geq C_G\varp_0(x)\varp_0(y), \end{eqnarray}

    yielding,
                \begin{eqnarray}
             \xi_0(x)=  \int G(x,y)\,dy \geq C_G\varp_0(x)\int\varp_0(y)\, dy,
                \end{eqnarray}

 where  $\xi_0$  denote the solution of  $L_0\xi_0 =1$.

In this paper    we consider the fractional Schr\"{o}diner operators
$$L_V:=L_0-V,~~V\in L^1_{Loc}(\Om).$$
  In particular the case
  \begin{eqnarray}\label{pot0f2}
  V(x)=\frac{c}{|x|^\alpha},\ 0<c\leq
  c^*:=\frac{2^\alpha\Gamma^2(\frac{d+\alpha}{4})}{\Gamma^2(\frac{d-\alpha}{4})},
  \end{eqnarray}
  \begin{eqnarray}\label{pot1f2}
  V(x)=\frac{c}{\delta^{\alpha}},\ 0<c\leq c^*:=\frac{\Gamma^2(\frac{\alpha+1}{2})}{\pi},
  \end{eqnarray}
  where $\delta$ is the Euclidian distance function between $x$ and $\Om^c$, $d\geq 3$ and $\Om$ is regular (see \cite{fmt13,frank}).\\
We shall prove that under some realistic assumptions, and especially under the assumptions
that some improved Sobolev and Hardy-type inequalities hold true,  then The operator $L_V$ still has
discrete spectrum, a unique normalized ground state $\varp_0^V>0\ a.e.$  Furthermore $\varp_0^V$  is
comparable to the solution, $\xi_0^V$ of the equation $L_V\xi_0^V=1$ (i.e., comparable to $L_V^{-1}1$). In
other words, if we designate by $G^V$ the Green's kernel of $L_V$, then
\begin{eqnarray}
\phiv\sim \xi_0^V=\int_\Om G^V(x,y)\,dy  ~~{\rm if}\ \Om\ {\rm is}\ C^{1,1}.
\end{eqnarray}
We shall however,  prove that the intrinsic ultracontractivity property is still preserved. Namely, the
operator $e^{-tL_V},\ t>0$ is IUC for domains which are less regular than $C^{1,1}$ domains.\\

For $\alpha=2$ (the local case), various types of comparison results as well as pontwise estimates for
ground states of the Dirichlet-Schr\"odinger operator were obtained in \cite{zuazua,
dupaigne-nedev,davilla-dupaigne,cipriani-grillo,davies-book} and in \cite{benamor-osaka} for more general
potentials in the framework of (strongly) local Dirichlet. Whereas the preservation of the intrinsic
ultracontractivity can be found in \cite{banuelos91} for Kato potentials, in \cite{cipriani-grillo} and in
\cite{benamor-osaka} in the framework of (strongly) local Dirichlet.  The potentials satisfying (\ref{pot0f2})
are treated in \cite{ali-IOP}\\

%
%
Our method relies  basically on an improved Sobolev inequality together with  a transformation argument
(Doob's transformation) which leads to a generalized ground state representation.\\
The paper is organized as follows: In section2 we give the backgrounds together with some preparing results.
For the comparability of the ground states we shall consider two situations separately: the subcritical
(section3) and the critical case (section4).\\
To get the estimates for the ground states of the approximating operator we shall use on one side the
intrinsic ultracontractivity property and on the other side Moser's  iteration technique.
\section{Preparing results}
We first give some preliminary results that are necessary for the later development of the paper. Some of
them are known. However, for the convenience of the reader we shall give new proofs for them.\\
Let $0<\alpha<\min(2,d)$. Consider the quadratic form $\calE^\alpha$ defined in $L^2:=L^2(\R^d,dx)$ by
\begin{eqnarray}
\calE^\alpha(f,g)&=&\frac{1}{2}\calA{(d,\alpha)}\il_{\R^d}\il_{\R^d}
\frac{(f(x)-f(y))(g(x)-g(y))}{|x-y|^{d+\alpha}}\,dxdy,\nonumber\\
D(\calE^\alpha)&=&W^{\alpha/2,2}(\R^d)
=\{  f\in L^2(\R^d)\colon\,\calE^\alpha[f]:=\calE^\alpha(f,f)<\infty\},\,
\end{eqnarray}
where
\begin{eqnarray}
\calA{(d,\alpha)}=\frac{\alpha\Gamma(\frac{d+\alpha}{2})}
{2^{1-\alpha}\pi^{d/2}\Gamma(1-\frac{\alpha}{2})}.
\label{analfaf2}
\end{eqnarray}
It is well known that $\calE^\alpha$ is a transient Dirichlet form and  is related (via Kato representation
theorem) to the selfadjoint operator, commonly named the  $\alpha$-fractional Laplacian on  $\R^d$ which we
shall denote by  $ (-\Delta)^{\alpha/2}$.\\
Alternatively, the expression of the operator  $(-\Delta)^{\alpha/2}$ is given by (see
\cite[Eq.3.11]{bogdan})
\begin{equation}\label{Lapff2}
    (-\Delta)^{\alpha/2}f(x)= \calA{(d,\alpha)} \lim_{\epsilon\rightarrow 0^+}\int_{\{y\in{\mathbb{R}^d},
    |y-x|>\epsilon\}} {\frac{f(x)-f(y)}{|x-y|^{d+\alpha}} dy},
\end{equation}
provided the limit exists and is finite.\\
From now on we shall ignore in the notations the dependence on $\alpha$ and shall set $\int\cdots$ as a
shorthand for $\int_{\R^d}\cdots$. The notation q.e. means quasi everywhere with respect to the capacity
induced by $\calE$.\\
For every open  subset  $\Om\subset\R^d,$ we  denote by $L_0:=(-\Delta)^{\alpha/2}|_\Om$ the localization of
$(-\Delta)^{\alpha/2}$ on $\Om$, i.e., the operator which Dirichlet form in $L^2(\overline\Om,dx)$ is given
by
\begin{eqnarray*}
D(\calE)&=&W_0^{\alpha/2}(\Om)\colon=\{f\in W^{\alpha/2,2}(\R^d)\colon\, f=0 ~~~\calE-q.
e.~on~\Om^c\}\nonumber\\
\calE(f,g)&=&\frac{1}{2}\calA{(d,\alpha)}\int
\int\frac{(f(x)-f(y))(g(x)-g(y))}{|x-y|^{d+\alpha}}\,dxdy\nonumber\\
&=&\frac{1}{2}\calA{(d,\alpha)}\big(\int_\Om\int_\Om
\frac{(f(x)-f(y))(g(x)-g(y))}{|x-y|^{d+\alpha}}\,dxdy\nonumber\\&+&\int_\Om f(x)g(x)\kappa_\Om^{(\alpha)}(x)\,dx\big),
\forall\,f,g\in W_0^{\alpha/2}(\Om).
\end{eqnarray*}
where
\begin{eqnarray}
\kappa_\Om^{(\alpha)}(x):=\calA(d,\alpha)\int_{\Om^c}\frac{1}{|x-y|^{d+\alpha}}\,dy.
\end{eqnarray}
%
The Dirichlet form $\calE$ coincides with the closure of $\calE^\alpha$ restricted to $C_c^\infty(\Om)$, and is
therefore regular and furthermore transcient.\\
We also recall the known fact  that $L_0$ is irreducible even when $\Om$  is disconnected
\cite[p.93]{bogdan}.\\
If moreover $\Om$ is bounded, thanks to the well known Sobolev embedding,
\begin{eqnarray}\label{IS0f2}
\big(\int_{\Om}|f|^{\frac{2d}{d-\alpha}}\,dx\big)^{\frac{d-\alpha}{d}}\leq
C(\Om,d,\alpha)\calE_{}[f],\ \forall\,f\in\Wo,
\label{sobo-freef2}
\end{eqnarray}
the operator $L_0$ has compact resolvent (that we shall denote by $K:=L_0^{-1}$) which together with the
irreducibility property imply that there is a unique continuous bounded, $L^2(\Om,dx)$ normalized function
$\varp_0>0$ and $\lam_0>0$ such that
\begin{eqnarray}
L_0\varp_0=\lam_0\varp_0\ {\rm on}\ \Om.
\end{eqnarray}
We shall prove that this property of $L_0$ is still preserved by perturbations of the form $V\in L^1_{Loc}$.
However, singularities will appear for the ground state of the perturbed operator provided $\Om$ contains
the origin.\\

Let $V_*$ be a fixed positive potentials such  that $V_*\in L^1_{loc}(\Omega)$, we shall  also adopt  some
assumptions along the paper.\\
The  first  assumption is the following Hardy-type inequality : There is a finite constant $C_H>0$ such
that
\begin{eqnarray}
\int\frac{f^2(x)}{\varp_0^2(x)}\,dx\leq C_H\calE[f],\ \forall\,f\in\Wo.
\label{hardyf2}
\end{eqnarray}
\begin{rk}{\rm
The latter inequality holds true for bounded domains satisfying the uniform interior ball condition and
$d\geq 2,\ \alp\neq 1$. Indeed for this class of   domains we already observed that
\begin{eqnarray}
\varp_0\geq C\delta^{\alpha/2},
\label{LBf2}
\end{eqnarray}
whereas \cite[Corollary 2.4]{song} asserts  that if $\Om$ is a Lipschitz domain then  for every $\alp\neq 1$
and $d\geq 2$ we have
\begin{eqnarray}
\int\frac{f^2(x)}{\delta^\alpha(x)}\,dx\leq C_H\calE[f],\ \forall\,f\in\Wo,
\label{LB2f2}
\end{eqnarray}
Combining the two inequalities yields (\ref{hardyf2}).

}

\end{rk}

~~\\

\section{The subcritical case}
In this section we fix:\\
                A positive potentials  $V\in L^1_{loc}(\Om),$  such that there is $\kappa\in(0,1)$, with
\begin{eqnarray}
\int f^2(x)V(x)\,dx\leq\kappa\calE[f],\ \forall\,f\in W_0^{\alpha/2,2}(\Om).
\end{eqnarray}

Having in mind that $0<\kappa<1$, we conclude that the quadratic form which we denote by $\calE_V$ and which
is defined by
\begin{eqnarray}
D(\calE_V)=W_0^{\alpha/2,2}(\Om),\  \calE_V [f]=\calE [f]-\int f^2V\,dx,\ \forall\,f\in
W_0^{\alpha/2,2}(\Om),
\end{eqnarray}
is closed in $L^2(\Om,dx)$ and is even comparable to $\calE$. Hence setting $L_V$ the positive selfadjoint
operator associated to $\calE_V$, we conclude that $L_V$ has purely discrete spectrum
$0<\lam_0^V<\lam_1^V<\cdots<\lam_k^V\to\infty$, as well.\\
Furthermore the associated semigroup $e^{-tL_V},\ t>0$ is irreducible (it has a kernel which dominates the
heat kernel of the free operator $L_0$). Thereby there is a unique $\varp_0^V\in\Wo$ such that
\begin{eqnarray}
\|\varp_0^V\|_{L^2}=1,\ \phiv> 0\ q.e.\ {\rm and}\ L_V\varp_0^V=\lam_0^V\varp_0^V.
\end{eqnarray}

            Two real-valued, measurable  a.e. positive and essentially bounded functions $S$ and $F$ on $\R^d$
            such that either $S\neq 0$ or $F\neq 0$. Let $w\in W_0^{\alpha/2,2}(\Om),$ we say that $w$ is a
            solution of the equation
\begin{eqnarray}
L_V w=Sw+F,
\end{eqnarray}
if
\begin{eqnarray}
\calE_{V}(w,f)=\int fSw\,dx+\int fF,\ \forall\,f\in W_0^{\alpha/2,2}(\Om).
\end{eqnarray}

In the goal of obtaining the precise behavior of the ground state, we proceed to transform the form
$\calE_V$ into  a Dirichlet form on $L^2(\Om,w^2dx)$, where $w>0$ q.e.  is a  solution of the equation $L_V
w=Sw+F$.\\
Let $Q^w$ be the $w$-transform of $L_V-S$, i.e.,  the quadratic form defined in $L^2(\Om,w^2dx)$ by
\begin{eqnarray}
D(Q^w):=\{f\colon\,wf\in\Wo\}\subset L^2(\Om,w^2dx),\\ \ Q^w[f]=\calE_V^w[f]-\int w^2f^2S\,dx,\
\forall\,f\in\,D(Q^w)~~where~~\calE_V^w[f]=\calE_V[wf].
\end{eqnarray}
\begin{lema} The form $Q^w$ is a regular Dirichlet form and
\begin{eqnarray}
Q^w[f]=\cf\il\il \frac{(f(x)-f(y))^2}{|x-y|^{d+\alpha}} w(x)w(y)\,dxdy +\int f^2Fw,\,\ \forall\,f\in
D(Q^w).
\label{IDf2}
\end{eqnarray}
\label{closabilityf2}
\end{lema}

\begin{proof} Obviously $Q^w$ is closed and densely defined as it is unitary equivalent to the closed
densely defined form $\calE_V^w$. Let us prove (\ref{IDf2}).\\
Writing
\begin{eqnarray}
w(x)w(y)\big(\frac{g(x)}{w(x)}-\frac{g(y)}{w(y)}\big)^2&=&(g(x)-g(y))^2+g^2(x)\frac{(w(y)-w(x))}{w(x)}\nonumber\\
&&+g^2(y) \frac{w(x)-w(y)}{w(y)},
\end{eqnarray}
and setting $g=wf$,  we get
\begin{eqnarray}
Q^w[f]&=&\cf\il_{}\il_{}\frac{(f(x)-f(y))^2}{|x-y|^{d+\alpha}}w(x)w(y)\,dx\,dy\nonumber\\
&+&\calA{(d,\alpha)}\il_{}\il \frac{w(x)-w(y)}{|x-y|^{d+\alpha}}{f^2(x)}{w(x)}\,dx\,dy\nonumber\\
 &-&\il{}f^2(x)w^2(x)V(x)\,dx\nonumber\\&-&\il{}f^2(x)w^2(x)S(x),\ \forall\,f\in D(Q^w),\nonumber\\
&\geq&\calA{(d,\alpha)}\il_{}\il \frac{w(x)-w(y)}{|x-y|^{d+\alpha}}{f^2(x)}{w(x)}\,dx\,dy\nonumber\\
 &-&\il{}f^2(x)w^2(x)V(x)\,dx\nonumber\\&-&\il{}f^2(x)w^2(x)S(x)\,dx\ \forall\,f\in D(Q^w),
\end{eqnarray}
we derive in particular that the integral
\begin{eqnarray}
\il_{}\il \frac{w(x)-w(y)}{|x-y|^{d+\alpha}}{f^2(x)}{w(x)}\,dx\,dy\ \ {\rm is\ finite}.
\end{eqnarray}
Thus using Fubini's together with dominated convergence theorem, we achieve
\begin{eqnarray}
Q^w[f]&=&\cf\il_{}\il_{}\frac{(f(x)-f(y))^2}{|x-y|^{d+\alpha}}w(x)w(y)\,dx\,dy\nonumber\\
&+& \calA{(d,\alpha)}\il_{}{f^2(x)}{w(x)}\big(\lim_{\epsilon\rightarrow0}
\il_{\{|x-y|>\epsilon\}}\frac{w(x)-w(y)}{|x-y|^{d+\alpha}}\,dy\big)\,dx\nonumber\\
 \nonumber\\
 &-&\il{}f^2(x)w^2(x)V(x)\,dx\nonumber\\&-&\il{}f^2(x)w^2(x)S(x)\,dx\ \forall\,f\in D(Q^w).
\label{AVD1f2}
\end{eqnarray}
Now, owing to  the fact that $w$ is a solution of the equation
\begin{eqnarray}
L_V w=Sw+F,
\end{eqnarray}
having (\ref{Lapff2}) in hands and  substituting in  (\ref{AVD1f2}) we get formula (\ref{IDf2}) from which we read
that $Q^w$ is Markovian  and hence a Dirichlet form.\\
{\em Regularity}:  Relying on the  expression (\ref{IDf2}) of $Q$, we learn from \cite[Example
1.2.1.]{fuku-oshima}, that $C_c^\infty(\Om)\subset D(Q^w)$ if and only if
\begin{eqnarray}
J:=\int_\Om\int_\Om\frac{|x-y|^2}{|x-y|^{d+\alpha}}w(x)w(y)\,dx\,dy<\infty.
\end{eqnarray}
Set $r'=2-\alpha$.  Then $0<r'<d$.
We rewrite J as
\begin{eqnarray*}
 J: &=& \int_\Om\int_\Om\frac{w(x)w(y)}{|x-y|^{d-r'}}\,dx\,dy\\
  &\leq& \frac{1}{2}\int_\Om\int_\Om\frac{w(x)^2+w(y)^2}{|x-y|^{d-r'}}\,dx\,dy \\
   &=& \int_\Om\int_\Om\frac{w(x)^2}{|x-y|^{d-r'}}\,dx\,dy\\
  &=&  \int_\Om w(x)^2\big(\int_\Om\frac{\,dy}{|x-y|^{d-r'}}\big)\,dx<\infty,
\end{eqnarray*}

with
 $$\sup_{x\in\Om}\big(\int_\Om\frac{\,dy}{|x-y|^{d-r'}}\big)<\infty.$$

Hence  $J$ is finite. \\
Hence from the Beurling--Deny--LeJan formula (see \cite[Theorem 3.2.1, p.108]{fuku-oshima}) together with
the identity (\ref{IDf2}), we learn that $Q^w$ is regular, which completes the proof.
\end{proof}
We designate by $L^w$ the operator associated to $Q^w$ in the weighted Lebesgue space $L^2(\Om,w^2dx)$ and
$T_t^w,\ t>0$ its semigroup. Then
\begin{eqnarray}
 L^w=w^{-1}(L_V-S)w\ {\rm and}\ T_t^w=w^{-1}e^{-t(L_V-S)}w,\ t>0.
 \end{eqnarray}
%

In the sequel set:
$$C_0=C_G\int\phio(y)S(y)w(y)+C_G\int\phio(y)F(y),$$
\begin{eqnarray}
 r:=\frac{d}{d-\alpha}~~, \,A:= (C_0C_H+C_S)\big(1+\lambda_0 C_S|\Om|^{1-1/r}\big)~~ \hbox{and}~~q:=\frac{2r-1}{r}.
\end{eqnarray}

\begin{theo}
\label{TIS1f2}
For every $\,f\in D(Q^w)$, we have
$$(IS1)\qquad \parallel f^2\parallel_{ {L^q}(w^2dx)}\leq A\big(Q^{w}[f]+\int Sf^2w^2\big). $$
\label{main1f2}
\end{theo}

The proof of Theorem \ref{main1f2} relies upon auxiliary  results which we shall state in three lemmata.

\begin{lema} \label{formulaf2}The following identity holds true
\begin{eqnarray}
\phiv&=&K(V\phiv)+\lamv K\phiv\ a.e,
\end{eqnarray}
where $$K\varphi:=\int G(.,y)\varphi(y)\,dy.$$
\label{ground-repf2}
\end{lema}
\begin{proof}
Set
\begin{eqnarray}
u=\phiv-K(V\phiv)-\lamv K\phiv.
\end{eqnarray}
Owing to the fact that $\phiv$ lies in $\Wo$ and hence lies in $L^2(Vdx)$, we obtain that the  measure
$\phiv  V$ has finite energy integral with respect to the Dirichlet form $\calE_\Om$, i.e.,
\begin{eqnarray}
\int |f\phiv V|\,dx\leq \gamma(\calE[f])^{1/2},\ \forall\,f\in C_c^\infty(\Om),
\end{eqnarray}
and therefore  $K(V\phiv)\in\Wo$. Thus $u\in\Wo$ and satisfies the identity
\begin{eqnarray}
\calE(u,g)&=&\calE(\phiv,g)-\int\phiv gV\,dx
-\lamv\int\phiv g\,dx\nonumber\\
&=&\calE_{V}(\phiv,g)-\lamv\int\phiv g\,dx=0,\ \forall\,g\in\Wo.
\end{eqnarray}
Since $\calE$ is positive definite  we conclude that $u=0\, a.e.$, which yields the result.
\end{proof}

\begin{lema} Let $w$ be as in Theorem \ref{main1f2}. Then the  following inequality holds true
\begin{eqnarray}
w\geq C_0\phio\ q.e..
\end{eqnarray}
where  $$C_0=C_G\int\phio(y)S(y)w(y)+C_G\int\phio(y)F(y).$$
\label{w-lowerf2}
\end{lema}
\begin{proof} As in the proof of Lemma \ref{formulaf2} we show that $w$ satisfies
$$
w-KVw=KSw+KF,
$$

We also recall the known fact that since $T_t=e^{-L_0t}$ is IUC (see\cite{kul}), then there is a finite
constant $C_G$, such that
\begin{equation}\label{lowerf2}
  G(x,y)\geq C_G \phio(x)\phio(y),
\end{equation}

yielding,
$$
w\geq KSw+KF\geq C_G\phio \int\phio(y)S(y)w(y)+C_G\phio \int\phio(y)F(y)\ q.e..
$$
and
$$C_0=C_G\int\phio(y)S(y)w(y)+C_G\int\phio(y)F(y)$$
\end{proof}

\begin{lema} We have,
\begin{equation}\label{s2f2}
 \int f^2\leq C_0C_H \cf\frac{}{}\il\il \frac{(f(x)-f(y))^2}{|x-y|^{n+\alpha}} w(x)w(y)\,dxdy+C_0C_H\lam_0
 \int {w}^2f^2,\, \forall\,  f\in D(Q^w).
\end{equation}
\label{L2-estimatef2}
\end{lema}

\begin{proof}
At this stage we use Hardy's inequality (\ref{hardyf2}), which states that there
is a constant $C_H>0$  such that
\begin{equation}\label{s3f2} \int\frac{u^{2}}{{\phio}^2}\,dx \leq C_H\cf\il\il
\frac{(u(x)-u(y))^2}{|x-y|^{n+\alpha}}\,dxdy,\, \forall\,u\in \Wo.
 \end{equation}

Let $f\in D(Q^w)\subset D(Q^\phio)$. Taking $ u=f\phio$ in inequality  (\ref{s3f2}) yields
\begin{eqnarray}\label{S3f2}
\int f^2&=&\int\frac{f^2{\phio}^2}{{\phio}^2}\leq
 C_H\cf\int\int \frac{(f\phio(x)-f\phio(y))^2}{|x-y|^{n+\alpha}} dxdy\nonumber\\
   &=& C_H\cf\il\il\frac{ (f(x)-f(y))^2}{|x-y|^{n+\alpha}}\phio(x)\phio(y)dxdy\nonumber\\&+&C_H\cf\int\int
   \frac{(\phio(x)-\phio(y)))({f^2\phio}(x)-f^2{\phio}(y))}{|x-y|^{n+\alpha}} dxdy.
\end{eqnarray}

Thanks to the fact that $\phio$ is an eigenfunction associated to $\lam_0$, we achieve
\begin{equation}\label{Ssf2}
\cf\int\int \frac{(\phio(x)-\phio(y))({f^2\phio}(x)-f^2{\phio}(y))}{|x-y|^{n+\alpha}}dxdy= \lam_0
\int f^2{\phio}^2 .
\end{equation}
Combining  (\ref{Ssf2}) with (\ref{S3f2}) we obtain
\begin{eqnarray}\label{psif2}
 \int f^2&\leq& C_H\cf\il\il\frac{ (f(x)-f(y))^2}{|x-y|^{n+\alpha}}\phio(x)\phio(y)dxdy\nonumber\\
 &+&C_H\lam_0\int{}(\phio)^2f^2,~~\, \forall\,f\in D(Q^w).
\end{eqnarray}
Having the lower bound for $w$ given by Lemma \ref{w-lowerf2} in hand, we establish
\begin{eqnarray}
\int f^2&\leq& C_HC_0\cf\int\int
\frac{(f(x)-f(y))^2}{|x-y|^{n+\alpha}}w(x)w(y)dxdy\nonumber\\&+&C_HC_0\lam_0\int {w}^{2}f^2,\,
\forall\,f\in D(Q^w).
\end{eqnarray}

\end{proof}

\begin{lema}  Set
\begin{eqnarray}
\Lam_1=1+\frac{C_HC_0}{2},\ \Lam_2=\frac{\|F\|_{\infty}^2}{2}+\frac{C_HC_0\lam_0}{2},
\end{eqnarray}
$C_0$ being the constant appearing in Lemma \ref{w-lowerf2}. Then
\begin{equation}\label{c1f2}Q^{w}[f]\leq\Lambda_1\cf\int\int
\frac{(f(x)-f(y))^2}{|x-y|^{n+\alpha}}w(x)w(y)dxdy+\Lambda_2\int w^2f^2
,\ \forall\,f\ D(Q^w).
\end{equation}
\end{lema}
\begin{proof}
We have already established that
\begin{eqnarray}
Q^w[f]=\cf\int\int\frac{(f(x)-f(y))^2}{|x-y|^{n+\alpha}}w(x)w(y)dxdy+\int f^2Fw,\ \forall\,f\in D(Q^w).
\end{eqnarray}
Making use of H\"older's and Young's inequality together with inequality (\ref{s2f2}) we obtain

\begin{eqnarray*}
Q^{w}[f]&\leq&\cf \il\il\frac{(f(x)-f(y))^2}{|x-y|^{n+\alpha}}w(x)w(y)dxdy+\big(\int f^2\big)^{\frac{1}{2}}
\big(\int f^2F^2w^2\big)^{\frac{1}{2}}\nonumber\\
&\leq&\Lam_1\cf\il\il\frac{ (f(x)-f(y))^2}{|x-y|^{n+\alpha}}w(x)w(y)dxdy+\Lam_2\int f^2w^2,\ \forall\,f\in
D(Q^w),
\end{eqnarray*}
which finishes the proof.

\end{proof}

\begin{proof} {\em of Theorem \ref{main1f2}}. We observe first that
\begin{eqnarray}
Q^{w}[f]+\int Sf^2w^2=\calE_{\V}^w[f]:=\calE_{\V}[wf], \forall\,f\in D(Q^w).
\end{eqnarray}

 By H\"older's inequality, we get for every $f\in D(Q^w)$,
\begin{eqnarray}\label{l31f2}
  \int_{}w^2f^{2(2-1/r)}&\leq& \big(\int w^{2r}f^{2r}\big)^{1/r}\big(\int f^2\big)^{1-1/r}
\end{eqnarray}

Using that $\calE$ and $\calE_V$ are equivalent, and by the Sobelev inequality \ref{IS0f2}, then there exists a
finite constant positive $C_S$ and $r:=\frac{d}{d-\alpha}>1$ such that
\begin{eqnarray}\label{l32f2}
 \big(\int g^{2r}\big)^{1/r}\leq C_S\calE_{\V}[g],~~ \hbox{for all}~ g\in\Wo  .
\end{eqnarray}

Taking $g=wf$, we have

\begin{eqnarray}\label{l322f2}
 \big(\int w^{2r}f^{2r}\big)^{1/r}\leq C_S\calE_{\V}^w[f]  .
\end{eqnarray}
On the other hand we have, according to   Lemma \ref{L2-estimatef2}
\begin{eqnarray}
\int f^2\leq C\big(\cf\il\il \frac{(f(x)-f(y))^2}{|x-y|^{n+\alpha}}w(x)w(y)dxdy
+\lam_0\int w^2f^2\big).
\end{eqnarray}

Applying another time H\"older's inequality we get
\begin{eqnarray}
\int(fw)^2\leq |\Om|^{1-1/r}\|(fw)^2\|_{L^{r}}\leq C_S|\Om|^{1-1/r}
\calE_{\V}^w[f],\ \forall f\in D(Q^w).
\label{holderf2}
\end{eqnarray}
Recalling that $\calE_{\V}^w[f]\geq \cf\il\il \frac{(f(x)-f(y))^2}{|x-y|^{n+\alpha}}w(x)w(y)dxdy,$ we
achieve
\begin{eqnarray}\label{l33f2}
\int f^2\,\leq C_HC_0\big(1+\lam_0C_S|\Om|^{1-1/r}\big)\calE_{\V}^w[f] ,\ \forall f\in D(Q^w).
\end{eqnarray}
 Combining (\ref{l31f2}), (\ref{l32f2}) and (\ref{l33f2}), we get $(IS1).$

\end{proof}

For every $t>0$ we designate by $T_t^w$ the semigroup associated to the form $Q^w$ in the space
$L^2(w^2dx)$. We are yet ready to prove the ultracontractivity of   $T_t^w$.\\
Set
\begin{eqnarray}
s:=\frac{2}{q-1}:=\frac{2r}{r-1}.
\end{eqnarray}

\begin{theo}   Then $T_t^w$ is ultracontractive for every $t>0$ and there exists $C_1>0$ depends only on $A$
and $s$ such that
\begin{eqnarray}
\|T_t^w\|_{L^1(w^2dx),L^{\infty}}\leq C_1 t^{-s/2}e^{\|S\|_{\infty}t},\ \forall t>0.
\end{eqnarray}
\label{UCf2}
\end{theo}

\begin{proof}From Theorem \ref{main1f2}, we derive
\begin{eqnarray}
\parallel f^2\parallel_{L^{r}(w^2dx)}\leq A\big(Q^{w}[f]+\|S\|_{\infty}\int_{\Om}f^2w^2\,dx\big),\
\forall\,f\in D(Q^w).
\end{eqnarray}

  Since $Q^w$ is a Dirichlet form, it is known that a Sobolev embedding for the domain of a Dirichlet form
  yields the ultracontractivity of the related semigroup ( see \cite[Theorems 4.1.2,4.1.3]{saloff-coste}),
  which ends the proof.

\end{proof}

We shall apply Theorem \ref{main1f2}, to the special cases $V=0 , F=1$

\begin{theo} Let $\xiv:=L_V^{-1}1$, then
\begin{eqnarray}
\phiv\leq C(\V,t)\xi^{\V} ,\ a.e.\ \forall\,t>0,
\end{eqnarray}
where
\begin{eqnarray}
C(\V,t):=C_1 t^{-s/2}e^{t\lam_0}\, \forall\,t>0.
\end{eqnarray}
\label{comparison1f2}
\end{theo}

\begin{proof}   Applying Theorem \ref{main1f2} to the case $V=0, F=1$, we get $w=\xiv$, and it yields that the
semi-group $T_t^{\xiv}$ is ultracontractive and $\frac{\phiv}{\xiv}$ is an eigenfunction for $T_t^{\xiv}$
associated to the eigenvalue $e^{-t\lam_0},\ \forall\,t>0$. Thus
\begin{eqnarray}
\|\frac{\phiv}{\xiv}\|_{\infty}&\leq& e^{t\lam_0}\|T_t^{\xiv}\|_{L^2((\xiv)^2dm),L^{\infty}}\nonumber\\
& &\leq C_1 t^{-s/2}e^{t\lam_0},\ \forall t>0,
\end{eqnarray}
and
\begin{eqnarray}
\phiv\leq C_1 t^{-s/2}e^{t\lam_0}\xiv,\ a.e.\ \forall\,t>0,
\end{eqnarray}
which was to be proved.
\end{proof}

While for the upper pointwise estimate we exploited the idea of intrinsic ultracontractivity, for the
reversed estimate
we shall however, make use of Moser's iteration technique.

\begin{theo} For every $t>0$, the  following estimate holds true

\begin{equation}\label{GDf2}
\xiv\leq (A C_0 C_H C^2(V,t)+2\lamv)\phiv,\ a.e.,
\end{equation}
where
\begin{eqnarray}
C(\V,t):=C_1 t^{-s/2}e^{t\lam_0}\, \forall\,t>0.
\end{eqnarray}
\label{comparison2f2}
\end{theo}

For the proof we establish the following lemma:
\begin{lema}

Assume that $V\in L^\infty(\Om)$, then (\ref{GDf2}) holds true
\end{lema}
\begin{proof}
\textbf{Step1: Iteration formula}

          We claim that, there exists $q>1$ such that for all $j\geq1$, we have
           \begin{equation}\label{iteratef2}
(\int \rho_{}^{2jq}(\varphi_{0}^{\V})^{2}\,dx)^{\frac{1}{q}}\leq
 (ACC^2(\V,t)+2\lamv) j^2\int_{}\rho_{}^{2j}(\varphi_{0}^{\V})^{2}\,dx.\\
\end{equation}

~~

Consider the family of smooth domains

$$\Om_{\epsilon}=\{x\in\Om / dist(x,\partial\Om) > {\epsilon}\}.$$
Let $\xiv_{\epsilon}\in \Wa(\Om_{\epsilon})$ be the solution of $L_V\xiv_{\epsilon}=1$ in $\Om_{\epsilon}$.
Since $V\in L^{\infty}$, then by \cite[Lemma 4.4]{benamor-osaka}, the function $\xiv_{\epsilon}\in
\Wa(\Om_{\epsilon})\cap L^{\infty}(\Om_{\epsilon})$, increasing   and convergent. We assume that
$\xiv_{\epsilon}\nearrow u$ as $\epsilon\rightarrow0$ ( uniformly, $\xiv_{\epsilon}\in
C^{\alpha/2}(\overline{\Om_\epsilon}$)).   On the other hand, we
have $\xiv_{\epsilon}=K(V\xiv_{\epsilon})+K(1_{\Om_{\epsilon}})$ converge to $u=K(Vu)+K(1)$, thus $L_V u=1$
and by  unicity of solution we have
$u=\xiv$.

 Letting
\begin{eqnarray}
\rho_{\epsilon}:=\frac{\xiv_{\epsilon}}{\phiv}.
\end{eqnarray}

 Since $\phiv>0$ in $\Om$  there exists  $C_{\epsilon}>0$ such that $\phiv>C_{\epsilon}$ in
 $\Om_{\epsilon}$,
 it follows that $\rho_{\epsilon}\in\Wa(\Om_{\epsilon})\cap L^{\infty}(\Om_{\epsilon})$ and
 $$\rho_{\epsilon}\nearrow\rho:=\frac{\xiv}{\phiv}.$$

  Now  using
\begin{eqnarray*}
  w(x)w(y)\big(\frac{g_1(x)}{w(x)}-\frac{g_1(y)}{w(y)}\big)\big(\frac{g_2(x)}{w(x)}-\frac{g_2(y)}{w(y)}\big)=
  (g_1(x)-g_1(y))(g_2(x)-g_2(y)) \\
   -(w(x)-w(y))\big[\frac{g_1(x)g_2(x)}{w(x)}-\frac{g_1(x)g_2(x)}{w(x)}\big],
\end{eqnarray*}
with the equations satisfied by the ground state $\phiv$ and the function $\xiv_{\epsilon}$,
setting $g_1=\phiv f,$ $g_2=\xiv_{\epsilon}$ and $w=\phiv$
, we find,  for every $f\in \Wa(\Om_{\epsilon})\cap L^{\infty}(\Om_{\epsilon})$

\begin{eqnarray}
\cf\il_{}\il_{}\frac{(f(x)-f(y))(\rho_{\epsilon}(x)-\rho_{\epsilon}(y))}{|x-y|^{d+\alpha}}\phiv(x)\phiv(y)\,dx\,dy&=~~~~~~~\nonumber\\
\cf\il_{}\il_{}\frac{(\phiv f(x)-\phiv
f(y))(\xiv_{\epsilon}(x)-\xiv_{\epsilon}(y))}{|x-y|^{d+\alpha}}\,dx\,dy&\nonumber\\
-\cf\il_{}\il_{}\frac{(\xiv_{\epsilon}
f(x)-\xiv_{\epsilon}f(y))(\phiv(x)-\phiv(y))}{|x-y|^{d+\alpha}}\,dx\,dy&=\nonumber\\
\calA{(d,\alpha)}\il_{}{\phiv f(x)}\big(\lim_{\epsilon'\rightarrow0}
\il_{\{|x-y|>\epsilon'\}}\frac{\xiv_{\epsilon}(x)-\xiv_{\epsilon}(y)}{|x-y|^{d+\alpha}}\,dy\big)\,dx -\il f
\xiv_{\epsilon}\phiv V(x)\,dx& \nonumber\\
-\bigg(\calA{(d,\alpha)}\il_{}{\xiv_{\epsilon} f(x)}\big(\lim_{\epsilon'\rightarrow0}
\il_{\{|x-y|>\epsilon'\}}\frac{\phiv(x)-\phiv(y)}{|x-y|^{d+\alpha}}\,dy\big)\,dx -\il f \xiv_{\epsilon}\phiv
V(x)\,dx \bigg)&= \nonumber\\
\il_{}{\phiv f(x)}\,dx-\lamv \il_{}{\xiv_{\epsilon}\phiv f(x)}\,dx .&
\end{eqnarray}

 Testing the latter  equation with $f=\rho_{\epsilon}^{2j-1}$, $j\geq 1$, ($f\in \Wa(\Om_\epsilon)\cap
 L^{\infty}(\Om_\epsilon)$
 ),  we deduce

\begin{eqnarray}\label{Eq1f2}
\cf\il_{}\il_{}\frac{(\rho_{\epsilon}^{2j-1}(x)-\rho_{\epsilon}^{2j-1}(y))(\rho_{\epsilon}(x)-\rho_{\epsilon}(y))}{|x-y|^{d+\alpha}}\phiv(x)\phiv(y)\,dx\,dy=
\nonumber\\\il_{}\rho_{\epsilon}^{2j-1}(x)({\phiv(x) }\,dx-\lamv{\xiv_{\epsilon}\phiv (x)})\,dx .
\end{eqnarray}

Using that for all $a,b\geq0 $  and $j\geq1$ we have
\begin{eqnarray}
(a^j-b^j)^2&:=&(a-b)^2(a^{j-1}+a^{j-2}b^{}+a^{j-3}b^{2}+......... +b^{j-1})^2 \nonumber\\
&\leq& j(a-b)^2(a^{2j-2}+a^{2j-4}b^{2}+a^{2j-6}b^{4}+......... +b^{2j-2})\nonumber\\
&\leq& j(a-b)^2(a^{2j-2}+a^{2j-3}b^{}+a^{2j-4}b^{2}+......... +b^{2j-2})\nonumber\\
&=& j(a-b)(a^{2j-1}-b^{2j-1})\label{Eq2f2}
\end{eqnarray}

which yields, using (\ref{Eq1f2}) and (\ref{Eq2f2})

\begin{eqnarray}\cf\il_{}\il_{}\frac{(\rho_{\epsilon}^{j}(x)-\rho_{\epsilon}^{j}(y))^2}{|x-y|^{d+\alpha}}\phiv(x)\phiv(y)\,dx\,dy
\leq \nonumber\\j\il_{}\rho_{\epsilon}^{2j-1}(x)({\phiv(x) }\,dx-\lamv{\xiv_{\epsilon}\phiv (x)})\,dx \leq
j\il{}\rho_{\epsilon}^{2j-1}(x)\phiv(x)\,dx.
\end{eqnarray}

According to Theorem \ref{comparison1f2}, we obtain

\begin{equation}\label{GD1f2}
 \cf\il_{}\il_{}\frac{(\rho_{\epsilon}^{j}(x)-\rho_{\epsilon}^{j}(y))^2}{|x-y|^{d+\alpha}}\phiv(x)\phiv(y)\,dx\,dy\leq
  C(\V,t)j\il{}\rho_{\epsilon}^{2j-1}(x)\rho\phiv(x)\,dx.
\end{equation}

Using H\"older inequality and Lemma \ref{L2-estimatef2} ( with $S=\lamv, F=0, w=\phiv$ and
$f=\rho_{\epsilon}^j$), it follows from (\ref{GD1f2})
that
\begin{eqnarray}
 \cf\il_{}\il_{}\frac{(\rho_{\epsilon}^{j}(x)-\rho_{\epsilon}^{j}(y))^2}{|x-y|^{d+\alpha}}\phiv(x)\phiv(y)\,dx\,dy\leq
 \nonumber\\
C(\V,t)j\big(\il{}(\phiv)^{2}\rho_{\epsilon}^{2j-2}\rho^2\,dx\big)^{\frac{1}{2}}(\il{}\rho_{\epsilon}^{2j}\,dx)^{\frac{1}{2}}\leq
\nonumber\\
C^{1/2}C(\V,t)j\big(\il_{}(\phiv)^{2}\rho_{\epsilon}^{2j-2}\rho^2\,dx\big)^{\frac{1}{2}}
\nonumber\\\times\bigg(\cf\il_{}\il_{}\frac{(\rho_{\epsilon}^{j}(x)-\rho_{\epsilon}^{j}(y))^2}{|x-y|^{d+\alpha}}\phiv(x)\phiv(y)\,dx\,dy+
\lamv\il{}(\phiv)^{2}\rho_{\epsilon}^{2j}\,dx\bigg)^{\frac{1}{2}}.
\end{eqnarray}

By Young's inequality, we obtain
\begin{eqnarray}
 \cf\il_{}\il_{}\frac{(\rho_{\epsilon}^{j}(x)-\rho_{\epsilon}^{j}(y))^2}{|x-y|^{d+\alpha}}\phiv(x)\phiv(y)\,dx\,dy\leq
 \nonumber\\
\frac{1}{2}CC^2(V,t)j^2\il_{}(\phiv)^{2}\rho_{\epsilon}^{2j-2}\rho^2\,dx+
\frac{\lambda_0^{(\V)}}{2}\il_{}(\phiv)^{2}\rho_{\epsilon}^{2j}\nonumber\\
\frac{1}{2}
\cf\il_{}\il_{}\frac{(\rho_{\epsilon}^{j}(x)-\rho_{\epsilon}^{j}(y))^2}{|x-y|^{d+\alpha}}\phiv(x)\phiv(y)\,dx\,dy,
\end{eqnarray}

so that

\begin{eqnarray}
\cf\il_{}\il_{}\frac{(\rho_{\epsilon}^{j}(x)-\rho_{\epsilon}^{j}(y))^2}{|x-y|^{d+\alpha}}\phiv(x)\phiv(y)\,dx\,dy\leq
\nonumber\\
CC^2(V,t)j^2\il_{}(\phiv)^{2}\rho_{\epsilon}^{2j-2}\rho^2\,dx+
\lambda_0^{(\V)}\il_{}(\phiv)^{2}\rho_{\epsilon}^{2j}.\label{Eq3f2}
\end{eqnarray}

By Theorem
(\ref{TIS1f2}), with $S=\lam_0^{\V}, F=0, w=\phiv$ and $f=\rho_{\epsilon}^j$, we get from (IS1)

\begin{eqnarray}
\qquad\parallel\rho_{\epsilon}^{2j}\parallel_{L^{q}((\varphi_{0}^{\V})^{2}\,dx)}\leq
A\big(\cf\il_{}\il_{}\frac{(\rho_{\epsilon}^{j}(x)-\rho_{\epsilon}^{j}(y))^2}{|x-y|^{d+\alpha}}\phiv(x)\phiv(y)\,dx\,dy
\nonumber\\+
\lambda_0^{(\V)}\il_{}\rho_{\epsilon}^{2j}(\varphi_{0}^{\V})^{2}\,dx\big),
\end{eqnarray}

using (\ref{Eq3f2}),

\begin{eqnarray}
\qquad\parallel\rho_{\epsilon}^{2j}\parallel_{L^{q}((\varphi_{0}^{\V})^{2}\,dx)}\leq
ACC^2(\V,t)j^2\int(\varphi_{0}^{\V})^{2}\rho_{\epsilon}^{2j-2}\rho^2\,dx\nonumber\\
+2\lamv j^2\int(\varphi_{0}^{\V})^{2}\rho_{\epsilon}^{2j}\,dx.
\end{eqnarray}
Thus

\begin{equation}
\big(\int \rho_{\epsilon}^{2jq}(\varphi_{0}^{\V})^{2}\,dx)\big)^{\frac{1}{q}}\leq
 ACC^2(\V,t)j^2\int(\varphi_{0}^{\V})^{2}\rho_{\epsilon}^{2j-2}\rho^2\,dx
+2\lamv j^2\int(\varphi_{0}^{\V})^{2}\rho_{\epsilon}^{2j}\,dx.
\end{equation}
 It is then easy to pass to the limit as $\epsilon\rightarrow0$, using e.g monotone convergence to obtain
 (\ref{iteratef2}).


\textbf{Step 2} we show when $V$ is bounded that
\begin{eqnarray}\label{xvf2}
\xi^{\V}\leq M(\V,t)\varphi_{0}^{\V},\ \forall\,t>0,
\end{eqnarray}

               iterate (\ref{iteratef2}),  define $j_k=2q^k$ for
$k=0,1...$ and
\begin{eqnarray}
\Theta_k=\big(\int\rho_{}^{j_k}(\varphi_{0}^{\V})^{2}\,dx\big)^{\frac{1}{j_k}}\ {\rm and}\
M(\V,t):= (ACC^2(\V,t)+2\lamv).
\end{eqnarray}

Then (\ref{iteratef2}) can be written as

\begin{eqnarray}
\Theta_{k+1}\leq
(M(\V,t)(q)^{2k})^{\frac{1}{2(q)^k}}\Theta_k.
\end{eqnarray}

Using this recursively yields

\begin{eqnarray}
\Theta_k\leq
M(\V,t)\Theta_0=M(\V,t)(\int\xi_{0}^{\V})^{2}\,dx)^{\frac{1}{2}}\leq M(\V,t),
\end{eqnarray}

for all $k=0,1,..$ .  Since the right-hand-side of the latter inequality is independent from $k$, we deduce

\begin{eqnarray}
\displaystyle\lim_{k\rightarrow\infty}\Theta_k=
\displaystyle\sup_{\Om}\rho\leq M(\V,t),
\end{eqnarray}
and this shows (\ref{xvf2}).
\end{proof}

\begin{proof}[Proof of Theorem 4.3.4]

 Let $$ V_k(x):=\min(V, k), k>0.$$  Then $L_k:=L_0-V_k$ increases in the strong resolvent sense to $L_V$. Since
 $L_V$ has compact resolvent, the latter convergence is even uniform (see \cite[Lemma 2.5]{bbb}).
Thus setting $\lam_0^{(k)}$'s the ground state energy  of the $L_k$'s , $\phiok$ its associated ground state
and $\xi_0^{(k)}:=L_k^{-1}1$ we obtain
\begin{eqnarray}
\lam_0^{(k)}\to\lam_0\ {\rm ,}\ \phiok\to\phiv\,{\rm and}~~ \xi_0^{(k)}\to\xiv\ {\rm in}\ L^2(\Om,dx).
\end{eqnarray}

Using Lemma 4.3.6, 4.3.3, 4.3.4 and Theorem 4.3.2 , it easy to be prove that\\
         $\lim_{k\to\infty}M(V_k,t)=M(V,t)\in(0,\infty)$ and
\begin{eqnarray}
 \xi^{(k)}\leq M(V_k,t)\varp_0^{(k)},\ a.e.\ \forall\,k\ {\rm large}.
\end{eqnarray}

Then $\xiv\leq M(V,t)\varp_0^{V},\ a.e.$.
which was to be proved.
\end{proof}
\section{The critical case}
The critical case differs in some respects from the subcritical one. The most apparent difference is that
the critical  quadratic form is no longer closed on the starting fractional Sobolev space $\Wo$.
Consequently the proof of Lemma \ref{ground-repf2} is no more valid to express the ground state for the simple
reason that it may not belong to $\Wo$. (See \cite{ali-IOP})\\
We shall however prove that the critical form is closable and has compact resolvent by mean of a
Doob's transformation. An approximation process will then lead to extend the identity of Lemma \ref{ground-repf2}
helping therefore to get the sharp estimate of the ground state.\\
The development of this section depends heavily on the following improved Sobolev inequality holds true:
there is a finite constant $C_S>0$ and $r>1$  such that
\begin{eqnarray}\label{ISI2f2}
(IS): &~~& \parallel \!f^2\!\parallel_{L^{r}}
\leq C_S\big(\mathcal{E}[f]-\int V_*  f^2(x)\,dx\big),\ \forall\,f\in C_c^{1}(\Omega).
\end{eqnarray}

\begin{rk}

 We observe  that if $d\geq3$ the  potentials (\ref{pot0f2}) and (\ref{pot1f2}) satisfy (\ref{ISI2f2}) with $r: =
 \frac{d}{d-\alpha}$ if $ 0<c< c^* $  and  with any $ 1<r<\frac{d}{d-\alpha}$ for $c= c^*$.
 (See \cite{frank,fmt13})

\end{rk}

           Hence we only consider  solutions that belong to the hilbert space $H$, defined as the completion
           of $C^{\infty}_c(\Om)$ with  respect to the norm
$${\|f\|^2}_H=\calE_\Om[f]-\int V_* f^2(x)\,dx.$$

   We denote by $H'$ the dual of $H$. Observe that $\Wo\subset H \subset  L^2(\Om).$

     If $F\in H'$ we say that $f\in H$ is solution of
    \begin{equation}\label{HSolf2}
      L_{ V_*}f=(L_0-{V_*})f=F
    \end{equation}
     if $$\calE_{V_*}(f,g)=\int_\Om Fg \,dx,~~~for~ all~g\in H.$$

\begin{lema}\label{lemsf2} Suppose (\ref{ISI2f2}) and let $F\in H'$. Then there exists a unique solution $f\in H$
which is a solution of (\ref{HSolf2}), and if $F\geq0$ in the sense of distributions then $f\geq0$ a.e.
\end{lema}
  \begin{proof} We can assume that $F>0$. It follows from Lax-Milgram
lemma that there exists a unique $f \in H$ such that
$$\calE_{V_*}(f,g)=\int_\Om Fg \,dx,~~~\forall~g\in H.$$

 We now show that $f\geq0$.  By definition of $H,$  there exists $f_k$ in $C^{\infty}_c(\Om)$  converging to
 $f$ in $H$. Letting $F_k=(-\Delta)^{\frac{\alpha}{2}} f_k -V_* f_k,$ it follows that $F_k\in H'$ and
 $F_k\rightarrow F$ in $H'.$

 Then by \cite[Lemma 3.3]{frank} $f_k\in W_0^{\alpha/2,2}(\Om)$, yielding that $f_k^-\in
 W_0^{\alpha/2,2}(\Om)$.\\
Activating Sobolev inequality (\ref{ISI2f2}) together with identity (\ref{Lapff2}) and utilizing the fact that
$F_k=(-\Delta)^{\frac{\alpha}{2}} f_k -V_* f_k,$  we obtain:
\begin{eqnarray}
 \|(f_k^-)^2\|_{H}&=&\big(\frac{1}{2}{\calA}(d,\alp)\int\int
 \frac{(f_k^{-}(x)-f_k^{-}(y))^2}{|x-y|^{d+\alpha}}\,dx\,dy-\int V_*(x){(f_k^{-})}^2(x)\,dx\big)\nonumber\\
  &\leq &-\bigg(\frac{1}{2}{\calA}(d,\alp)\int\int\frac{(f_k(x)-f_k(y))
  (f_k)^{-}(x)-f_k^{-}(y))}{|x-y|^{d+\alpha}}\,dx\,dy\nonumber\\
 &-&\int V_*(x)f_k(x)f_k^{-}(x)\,dx\bigg)\nonumber \\
&=&- \calE_{V_*}(f_k,f_k^-)\nonumber\\
 &=&-\int F_k f_k^{-}(x)\,dx\leq 0.\label{Eq4f2}
\end{eqnarray}
In the 'passage' from the first to the second inequality, we used the fact that for any Dirichlet form
${\mathcal{D}}$ one has ${\mathcal{D}}(f^+,f^-)\leq 0$ (See \cite[Theorem 4.4-i)]{roeckner-ma}), whereas the
equality before the last one is obtained with the help of the identity (\ref{Lapff2}).\\
To pass to  te limit in the last equation, we just  neet to prove that $f_k^-$ ramains boundes in $H$.

\begin{eqnarray}
 \|(f_k^-)\|^2_{H}&=&\frac{1}{2}{\calA}(d,\alp)\int\int
 \frac{(f_k^{-}(x)-f_k^{-}(y))^2}{|x-y|^{d+\alpha}}\,dx\,dy-\int V_*(x){(f_k^{-})}^2(x)\,dx\nonumber\\
  &= &\frac{1}{2}{\calA}(d,\alp)\int\int
 \frac{(f_k^{-}(x)-f_k^{-}(y))^2}{|x-y|^{d+\alpha}}\,dx\,dy-\int V_*(x){(f_k^{})}^2(x)\,dx\nonumber\\
 &+&\int V(x)(f_k^{+})^2(x)\,dx\nonumber \\
&\geq&\frac{1}{2}{\calA}(d,\alp)\int\int
 \frac{(f_k^{-}(x)-f_k^{-}(y))^2}{|x-y|^{d+\alpha}}\,dx\,dy-\int V_*(x){(f_k^{})}^2(x)\,dx\nonumber\\
 &+&\frac{1}{2}{\calA}(d,\alp)\int\int
 \frac{(f_k^{+}(x)-f_k^{+}(y))^2}{|x-y|^{d+\alpha}}\,dx\,dy\nonumber \\
&=&\frac{1}{2}{\calA}(d,\alp)\int\int
 \frac{(f_k^{}(x)-f_k^{}(y))^2}{|x-y|^{d+\alpha}}\,dx\,dy-\int
 V_*(x){(f_k^{})}^2(x)\,dx\nonumber\\&+&2\calE(f_k^{-},f_k^{+})\nonumber\\
   &\leq&\|(f_k)\|^2_{H} .\nonumber\\
\end{eqnarray}
Letting $k\rightarrow\infty$ in (\ref{Eq4f2}), we get
  $f^{-}\equiv 0$ in $\Om$ yielding $f\geq0.$
%


  \end{proof}

  Let $\dot\calE_*$ be the quadratic form defined by
\begin{eqnarray}
\calF:=D(\dot\calE_*)=\{ f\in H , L_{V_{*}}f\in L^2(\Om)\},\\\ \dot\calE_*[f]=\calE[f]-\int V_* f^2\,dx,\
\forall\,f\in \calF.
\end{eqnarray}
Let $S$ and $F$ be two real-valued, measurable  a.e. positive and essentially bounded functions  on $\R^d.$
  Let $ w_*\in H$ be solution of
   \begin{equation}\label{eqw2f2}
     L_{*} w_*=Sw_*+F.
   \end{equation}
   \begin{lema} There is a finite constant $\tilde{C_0}$ such that
\begin{eqnarray}
w_*\geq\tilde{C_0}\varp_0\ a.e.
\end{eqnarray}
\label{fisf2}
\end{lema}
\begin{proof}Let
$$ V_k(x) = min(V_{*}(x), k), k > 0,$$
$w_k$ is a solution of (\ref{eqw2f2}) with the potential $V_{*}(x)$ replaced by the potential $V_k(x)$, we
obtain
$$((-\Delta)^{\alpha/2}w_k,w_k)_{L^2(\Om)}\leq \|w_k\|_H +  \|V_k\|_{L^\infty(\Om)}
(w_k,w_k)_{L^2(\Om)}<\infty.$$

 Since  $(w_k)_k$ is bounded  in $\Wo$ and nondecreasing in k, it converges to  $w$ in $L^2(\Om)$
  and that ${w_k}$ remains bounded in $H$ so that $w\in H$ and   $w= w_*$. On the other hand by Lemma
  \ref{w-lowerf2} it yields
$$
w_k\geq C_G\phio \int\phio(y)S(y)w_k(y)+C_G\phio \int\phio(y)F(y)\ q.e.
$$

We note that here all the integrals are finite, and  that we can
pass to the limit in the equation satisfied by $w_k$ and conclude that
$$
w_*\geq C_G\phio \int\phio(y)S(y)w_*(y)+C_G\phio \int\phio(y)F(y)\ q.e..,
$$

\end{proof}
   By analogy to the subcritical case we define the $w_*$-transform of $\dot\calE_*$ which we denote by
   $\dot Q_*$ and is defined by
\begin{eqnarray}
D(\dot Q_*):=\{f\colon\,w_*f\in\calF\}\subset L^2(\Om,w_*^2dx),\nonumber\\  \dot Q_*[f]=\dot\calE_*[w_*f]-\int w_*^2f^2S\,dx,\
\forall\,f\in\,D(\dot Q_*).
\end{eqnarray}
Following the computations made in the proof of Lemma \ref{closabilityf2} we realize that $\dot Q_*$ has the
following representation
\begin{eqnarray}
\dot Q_*[f]=\frac{\mathcal{A}(d,\alp)}{2}\il\il \frac{(f(x)-f(y))^2}{|x-y|^{d+\alpha}} w_*(x)w_*(y)\,dxdy\nonumber\\+\int f^2Fw_*\,dx,\
\forall\,f\in D(\dot Q_*).~~~~~~~~~~~~~~~~~~~~~~~~
\end{eqnarray}
\begin{lema}
 The form $\dot Q_*$ is closable in $L^2(\Om,w_*^2dx)$. Furthermore its closure  is a Dirichlet form in
 $L^2(\Om,w_*^2dm)$. It follows, in particular that $\dot\calE_*$ is closable.
\label{closability2f2}
\end{lema}
\begin{proof} We first mention that since $\dot\calE_*$ is densely defined then $\dot Q_*$ is densely
defined as well.\\
Now we proceed to show that $\dot Q_*$ possesses a closed extension. To that end we introduce the form
$\tilde Q$ defined by
\begin{eqnarray}
&D(\tilde Q)&:=\big\{f\colon\,f\in L^2(\Om,w_*^2dx),\  \tilde Q[f]<\infty\big\}\nonumber\\
&\tilde Q[f]&=\frac{\mathcal{A}(d,\alp)}{2}\il\il \frac{(f(x)-f(y))^2}{|x-y|^{d+\alpha}}
w_*(x)w_*(y)\,dxdy\nonumber\\&&+\int f^2Fw_*\,dx,\,\ \forall\,f\in D(\tilde Q).
\end{eqnarray}
Arguing as in the proof of Lemma \ref{closabilityf2}, we obtain  that $C_c^{\infty}(\Om)\subset D(\tilde
Q)$.\\
Hence from the Beurling--Deny-LeJan formula (see \cite[Theorem 3.2.1, p.108]{fuku-oshima}), the form $\tilde
Q$ is the restriction to $C_c^{\infty}(\Om)$ of a Dirichlet form and therefore closable and Markovian. Since
$D(\dot Q_*)\subset D(\tilde Q)$ we conclude that $\dot Q_*$ is closable and Markovian as well, yielding
that its closure is a Dirichlet form. Now the closability of $\dot\calE_*$ is an immediate consequence of
the closability of $\dot Q_*$  which finishes the proof.
\end{proof}
From now on we set $\calE_*$  the closure of $\dot\calE_*$ and $L_*$ the selfadjoint operator related to
$\calE_*$, respectively $Q_*$ the closure of $\dot Q_*$ and $H_*$ its related selfadjoint operator. Finally
$T_t^*:=e^{-tL_*},\ t>0$ and $S_t:=e^{-tH_*},\ t>0$. Obviously $H_*=w_*^{-1}L_*w_*$.\\

Of course the  inequality (\ref{ISI2f2}) extends to the elements of $D(\calE_*)$ with $\dot\calE_*$ replaced
by $\calE_*$.
The idea of using improved Sobolev type inequality to get estimates for  the ground state was already used
in \cite{benamor-osaka, davilla-dupaigne}.

\begin{theo} For every $t>0$, the operator $S_t$ is ultracontractive. It follows that
\begin{itemize}
\item[i)]The operators $S_t,\ t>0$ and hence $T_t^*,\ t>0$ are  Hilbert-Schmidt operators and the operator
    $L_{*}$ has a  compact resolvent.
\label{spec-prop1f2}
\item[ii)] $ker(L_*-\lam_0^*)=\mathbb{R}\phis$ with $\phis> 0\ a.e.$
\item[iii)] If $\Omega$ satisfies the uniform interior ball condition then
\begin{equation}
\phis(x)\geq\big( C_G \lam_0^*\int\phio(y)^{\alpha/2}\phis(y)\,dy\big)\phio(x)^{\alpha/2},\ a.e.
\label{estimate-eigenfunctionf2}
\end{equation}
%
\end{itemize}
\label{ground-criticalf2}
\end{theo}
\begin{proof} The proof that $S_t,\ t>0$ is ultracontractive runs as the one corresponding to the
subcritical case with the help of Lemma \ref{closability2f2} and inequality (\ref{ISI2f2}) as main
ingredient.\\
i)  Every ultracontractive operator has an almost everywhere bounded kernel and since $w_*\in L^2(\Om)$ one
get that $S_t,\ t>0$ is a Hilbert-Schmidt operator as well as $T_t^*$ and hence $L_*$ has compact
resolvent.\\
ii) Since $T_t^*,\ t>0$ has a nonnegative kernel it is irreducible and the claim follows from the well know
fact that the generator of every irreducible semigroup has a nondegenerate ground state energy with a.e.
nonnegative ground state.\\
iii) The fact that $T_t^*$ is a Hilbert-Schmidt operator yields that $L_*$ possesses a Greeen kernel, $G_*$
and that $G_*\geq G$. Writing
\begin{eqnarray}
\phis=\lam_0^*\int G_*(\cdot,y)\phis(y)\,dy\geq\lam_0^*\int G(\cdot,y)\phis(y)\,dy
\end{eqnarray}
and using the lower bound (\ref{lowerf2}) yields the result.\\
\end{proof}

Let  $(V_k)$ be an increasing sequence of positive potentials such that $V_k\uparrow V_*$ and there is a
constant $0<\kappa_k<1$ such that for every $k\in\N$ we have
\begin{eqnarray}
\int_{}f^{2}\,V_k(x) dx \leq\kappa_k\calE[f],\ \forall\,f\in \mathcal{F}_.
\label{infinitesimalf2}
\end{eqnarray}
For example the sequence $V_k=(1-\frac{1}{k})V_*$ satisfies the above conditions.\\
By the assumption $0<\kappa_k<1$, we conclude that the  following forms
\begin{eqnarray*}
D(\calE_{V_k})=\calF,\ \calE_{V_k}[f]=\mathcal{E}[f]-\int_{\Om}f^2V_k(x)\,dx,\ \forall\,f\in\calF,
\end{eqnarray*}
are closed in $L^2$. For every integer $k$, we shall designate by $L_k$ the self-adjoint operator related to
$\calE_{V_k}$.\\
Let $0<V_k\uparrow V_*$, then $L_k:=L-V_k$, increases in the strong resolvent sense to  $L_*$.
 Since $L_*$ has compact resolvent, the latter convergence is even uniform (see \cite[Lemma 2.5]{bbb}). Thus
 setting $\lam_0^{(k)}$'s the ground state energy  of the $L_k$'s , $\phiok$ its associated ground state and
 $\xi_0^{(k)}:=L_k^{-1}1$ we obtain
\begin{eqnarray}
\lam_0^{(k)}\to\lam_0^*\ {\rm ,}\ \phiok\to\phis{\rm~~ and~~} \xi_0^{(k)}\to\xis \ {\rm in}\ L^2(\Om,dx).
\label{approxif2}
\end{eqnarray}
For an accurate description of the behavior of the ground state, we shall extend formula (\ref{ground-repf2})
to $\phis$.

Finally we resume.
\begin{theo} Let $V\in L^1_{loc}$ be a positive potential. Then under assumptions , (IS)and (HI)   the
following sharp estimate for the ground state $\phis$ holds true
\begin{eqnarray*}
\big((AC C_1^2\inf_{t>0} t^{-s}e^{2t\lamos}+1)\big)^{-1}\xis
\leq\phis\leq\xis( C_1\inf_{t>0}
 t^{-s/2}e^{t\lamos}),\ a.e.
\end{eqnarray*}
\label{sharp-comparisonf2}

\end{theo}

\begin{coro} We have

\begin{eqnarray}
\phis\sim\int G^{*}(\cdot,y)\,dy,\ a.e.
\end{eqnarray}
\label{endf2}
\end{coro}

We also derive by standard way the following large time asymptotics for the heat kernel.

\begin{coro} There is $T>0$ such that for every $t>T$,
\begin{eqnarray}
p_t^{*}(x,y)\sim e^{-\lamos t}\phis(x)\phis(y)\sim e^{-\lamos t}\xis(x)\xis(y),\,   a.e. .
\end{eqnarray}
It follows, in particular that
\begin{eqnarray}
 -\lamos=\lim_{t\to\infty}\frac{1}{t}\ln\big(\frac{p_t^{*}(x,y)} {\xis(x)\xis(y)}\big).
\end{eqnarray}
\end{coro}

~~~\\
\begin{tabular}{ll}
M.A.~Beldi &  \quad \\
Faculty of Sciences of Tunis, Tunisia &  \quad \\
{\small mohamedali.beldi@issatm.rnu.tn} &  \quad
{\small} \\
\end{tabular}

\bibliographystyle{alpha}

\begin{thebibliography}{biblio-frac-V1}

\bibitem[BAB11]{bbb}
H.~BelHadj~Ali, A.~Ben Amor and J.F.~ Brasche.
\newblock Large coupling convergence: Overview and new results.
\newblock In I.~Gohberg, Michael Demuth, Bert-Wolfgang Schulze, and Ingo Witt,
  editors, {\em Partial Differential Equations and Spectral Theory}, volume 211
  of {\em (Operator Theory: Advances and Applications)}, pages 73--117. Springer
  Basel, 2011.

\bibitem[Ba{\~n}91]{banuelos91}
R.~Ba{\~n}uelos.
\newblock Intrinsic ultracontractivity and eigenfunction estimates for
  {S}chr\"odinger operators.
\newblock {\em J. Funct. Anal.}, 100(1):181--206, 1991.

\bibitem[BB12]{benamor-forum2012}
N.~Belhaj~Rhouma and A.~Ben Amor.
\newblock {H}ardy's inequality in the scope of {D}irichlet forms.
\newblock {\em Forum Math.}, 24(4):751--767, 2012.

\bibitem[BBB013]{benamor-osaka}
A.~Beldi, N.~Belhaj~Rhouma and A.~Ben Amor.
\newblock Pointwise estimates for the ground states of some classes of
  positivity preserving operators.
\newblock {\em Osaka J. Math.}, 50: 765--793, 2013.



\bibitem[BBB13]{ali-IOP}
A.~Beldi, N.~Belhaj~Rhouma and A.~Ben Amor.
\newblock Pointwise estimates for the ground state of singular Dirichlet fractional Laplacian.
\newblock {Journal of Physics A: Mathematical and Theoretical.} Volume 46 Number 44, 2013.

\bibitem[BBC03]{bogdan}
K.~Bogdan, K.~Burdzy and Z.Q.~Chen.
\newblock Censored stable processes.
\newblock {\em Probab. Theory Related Fields}, 127(1):89--152, 2003.

\bibitem[BH86]{bliedtner}
J.~Bliedtner and W.~Hansen.
\newblock {\em Potential theory}.
\newblock Universitext. Springer-Verlag, Berlin, 1986.
\newblock An analytic and probabilistic approach to balayage.

\bibitem[CG98]{cipriani-grillo}
F.~Cipriani and G.~Grillo.
\newblock Pointwise properties of eigenfunctions and heat kernels of
  {D}irichlet-{S}chr\"odinger operators.
\newblock {\em Potential Anal.}, 8(2):101--126, 1998.

\bibitem[CKS10]{kim2010}
Z.Q.~Chen, P.~Kim and R.~Song.
\newblock Heat kernel estimates for the {D}irichlet fractional {L}aplacian.
\newblock {\em J. Eur. Math. Soc. (JEMS)}, 12(5):1307--1329, 2010.

\bibitem[CS03]{song}
Z.Q.~ Chen and R.~Song.
\newblock {H}ardy inequality for censored stable processes.
\newblock {\em Tohuku Math.J.}, 55, 2003.

\bibitem[Dav89]{davies-book}
E.B.~Davies.
\newblock {\em Heat kernels and spectral theory}.
\newblock Cambridge University Press, Cambridge, 1989.

\bibitem[DD03]{davilla-dupaigne}
J.~D{\'a}vila and L.~Dupaigne.
\newblock Comparison results for {PDE}s with a singular potential.
\newblock {\em Proc. Roy. Soc. Edinburgh Sect. A}, 133(1):61--83, 2003.

\bibitem[DN02]{dupaigne-nedev}
L.~Dupaigne and G.~Nedev.
\newblock Semilinear elliptic {PDE}'s with a singular potential.
\newblock {\em Adv. Differential Equations}, 7(8):973--1002, 2002.

\bibitem[Fit00]{fitz00}
P.J. ~Fitzsimmons.
\newblock {Hardy's inequality for Dirichlet forms.}
\newblock {\em J. Math. Anal. Appl.}, 250(2):548--560, 2000.

\bibitem[FLS08]{frank}
R.L.~Frank, E.H.~Lieb  and R.~Seiringer.
\newblock Hardy-{L}ieb-{T}hirring inequalities for fractional {S}chr\"odinger
  operators.
\newblock {\em J. Amer. Math. Soc.}, 21(4):925--950, 2008.

\bibitem[FMT13]{fmt13}
 {S.~Filippas}, {L.~Moschini}  and {A.~Tertikas},
   \newblock "{Sharp Trace Hardy-Sobolev-Maz'ya Inequalities and the Fractional Laplacian}",
 \newblock {Archive for Rational Mechanics and Analysis}, apr, 208,: 109--161,2013

\bibitem[F{\=O}T94]{fuku-oshima}
M.~Fukushima, Y.~{\=O}shima, and M.~Takeda.
\newblock {\em Dirichlet forms and symmetric {M}arkov processes}.
\newblock Walter de Gruyter \& Co., Berlin, 1994.

\bibitem[Han06]{Hansen-comp}
W.~Hansen.
\newblock Global comparison of perturbed {G}reen functions.
\newblock {\em Math. Ann.}, 334(3):643--678, 2006.

\bibitem[Kul98]{kul}
T.Kulczycki.
\newblock Intrinsic ultracontractivity for symmetric stable processes.
\newblock {\em Bull. Polish Acad. Sci. Math.}, 46(3):325--334, 1998.

\bibitem[MR92]{roeckner-ma}
Z.M.~ Ma and M.~R{\"o}ckner.
\newblock {\em Introduction to the theory of (nonsymmetric) {D}irichlet forms}.
\newblock Springer-Verlag, Berlin, 1992.

\bibitem[SC02]{saloff-coste}
L.~Saloff-Coste.
\newblock {\em Aspects of {S}obolev-type inequalities}, volume 289 of {\em
  London Mathematical Society Lecture Note Series}.
\newblock Cambridge University Press, Cambridge, 2002.

\bibitem[VZ00]{zuazua}
J.L~ Vazquez and E.~Zuazua.
\newblock The {H}ardy inequality and the asymptotic behaviour of the heat
  equation with an inverse-square potential.
\newblock {\em J. Funct. Anal.}, 173(1):103--153, 2000.

\bibitem[Yaf99]{yafaev}
D.~Yafaev.
\newblock Sharp constants in the {H}ardy-{R}ellich inequalities.
\newblock {\em J. Funct. Anal.}, 168(1):121--144, 1999.

\end{thebibliography}


\end{document}